\documentclass[11pt]{article}
\usepackage{amsmath,amssymb,amsthm,mathtools}
\usepackage{enumitem}
\usepackage{hyperref}
\usepackage{geometry}
\usepackage{microtype}
\usepackage{mathrsfs}
\usepackage{graphicx}
\usepackage{caption}
\usepackage{booktabs}
\usepackage{tikz}
\usetikzlibrary{arrows.meta,calc}
\title{Time Structure in Infinite Extensive Games}
\author{Shravan Luckraz \\ Schoolf of Economics and CeDEx China, University of Nottingham Ningbo China}
\date{\today}

\newtheorem{definition}{Definition}[section]

\newtheorem{theorem}[definition]{Theorem}
\newtheorem{lemma}[definition]{Lemma}
\newtheorem{proposition}[definition]{Proposition}
\newtheorem{corollary}[definition]{Corollary}
\newtheorem{remark}[definition]{Remark}
\newtheorem{example}[definition]{Example}

\begin{document}
\maketitle

\begin{abstract}
We extend graph‑theoretic characterizations of time‑structurable extensive games from finite informational digraphs to countable and continuum digraphs built on continuum rooted real trees. Introducing the symmetric‑class quotient poset of information sets as the canonical object an external clock must order, we develop a theory of transitive‑reduction for these quotients and prove existence and uniqueness under an interval‑finiteness condition. Under natural topological regularity assumptions like closed reachability classes, order‑density, and sigma‑compactness, we give necessary and sufficient conditions for real‑valued time labelings and construct measurable selection procedures in Polish spaces. We also provide constructive algorithms and a differential game example illustrating applications and computational implications.
\end{abstract}

\section{Introduction}
This paper extends graph‑theoretic characterizations of time‑structurable extensive games from finite informational digraphs to countable and continuum informational digraphs built on continuum rooted real trees, thereby bridging discrete and continuum models of temporal structure in extensive games. While prior work has emphasized equivalences between sequence‑based and graph‑based representations (Kline-Luckraz (2011)and Kline-Luckraz (2016)) and finite characterizations of time structure (Luckraz (2022 and Luckraz-Pansera 2023)), we show how those finite insights lift to richer informational continua by making the quotient/equivalence viewpoint central. Motivating examples from the literature on imperfect recall and nonstandard information structures (Piccione and Rubinstein 1997; Aumann, Hart, and Perry 1997) illustrate why a canonical, orderable object is needed when an external clock is imposed on an informational digraph. Algorithmic and ordering foundations (Kahn 1962) further motivate a formal treatment that combines combinatorial, topological, and measurable‑selection techniques.

At the heart of our approach is the symmetric class quotient poset of information sets, which we introduce as the canonical object any external clock must order. This quotient compresses informational redundancy while preserving the reachability relations that determine feasible temporal orderings; it therefore isolates the structural constraints that a global time labeling must respect. We develop a full theory of transitive reductions for these quotient posets, proving existence and uniqueness under an interval‑finiteness hypothesis that generalizes familiar finiteness conditions from the discrete literature. These results show how classical notions of acyclicity and minimality translate into the continuum setting and why naive continuum generalizations can fail without additional structural hypotheses.

For the continuum representation, we identify and impose explicit topological regularity conditions, like closedness of reachability classes, order‑density of the quotient, and sigma‑compactness of underlying trees and we show these conditions are both natural in applications and essentially sharp for our representation theorems. Under these hypotheses we obtain necessary and sufficient conditions for the existence of real‑valued time labelings in countable and continuum informational digraphs and we thus construct measurable selection procedures in Polish spaces that yield Borel or measurable time labelings when they exist. Alongside these existence results we provide constructive algorithms for recursively presented graphs and identify modeling pathologies that obstruct time structure.

Finally, the paper sharpens existing results on continuum representation statements by making topological hypotheses explicit. Our work also relates to computational complexity and equilibrium representations (von Stengel and Forges 2008) and to recent work on absorbing and clocked games (Hansen, Ibsen, and Neyman 2021). A detailed differential game example illustrates how continuum rooted real trees and symmetric equivalence class quotients arise naturally in continuous time strategic settings. 	Our results can also be used in pursuit-evasion games on graphs under imperfect visibility (Luckraz (2019), Luckraz (2023)).

The remainder of the paper is organized as follows: Section 2 sets notation and provides relevant background; Section 3 defines the quotient poset and states the main structural theorems; Sections 4–6 develop transitive‑reduction, labeling, and measurable‑selection results; Section 7 gives algorithms and diagnostics; Section 8 presents the differential‑game example; and Section 9 discusses implications and open problems.
This paper (i) formalizes collapsing cycles of length one and two into
equivalence classes, (ii) defines the quotient reachability order and
analyzes acyclicity there, (iii) extends the finite characterization to
countable and continuum informational digraphs, and (iv) provides
measurable and continuous representation theorems for Polish and
sigma‑compact informational spaces. We include algorithms, diagnostics,
and a worked differential‑game example.

\section{Model and notation}

\subsection{Informational digraph and reachability}
We follow the OR extensive form framework of Osborne and Rubinstein (1994), Kline Luckraz (2016) and the informational digraph definition of Pansera and Luckraz (2022).
Formally, let \(\mathcal H\) be the set of histories, \(\mathcal H^D\) decision
histories, \(N\) players, and \(\mathcal I=\bigcup_{i\in N}\mathcal I_i\)
the Kuhn information partition. For \(I',I''\in\mathcal I\) define the
precedence relation \(\prec\) by \(I'\prec I''\) iff there exist
histories \(h'\in I',h''\in I''\) with \(h'\) a proper initial segment of
\(h''\). The  informational digraph is \(G=(V,E)\) with \(V=\mathcal I\)
and \(E=\{(I',I''):I'\prec I''\}\).
\begin{definition}[Reachability]
For \(I,J\in\mathcal I\) write \(I R_G J\) if there exists a directed walk
from \(I\) to \(J\) in \(G\). The reachability relation \(R_G\) is the
transitive closure of \(E\) together with reflexivity.
\end{definition}

\begin{definition}[Transitive reduction]
A directed graph \(R=(X,F)\) is a \emph{transitive reduction} of the
reachability digraph \((X,\prec)\) if the transitive closure of \(R\)
equals \(\prec\) and \(F\) is minimal with this property.
\end{definition}

\begin{definition}[Cover relation]
For \(x,y\in X\) with \(x\prec y\) say \(y\) \emph{covers} \(x\) (write
\(x\lessdot y\)) if there is no \(z\in X\) with \(x\prec z\prec y\).
The cover relation \(C\subseteq X\times X\) is the set of all cover pairs.
\end{definition}

\begin{definition}[Interval‑finite poset]
The quotient poset \((X,\preceq)\) is \emph{interval‑finite} if for every
\(x\prec y\) the interval \([x,y]\) is finite.
\end{definition}

\begin{theorem}[Transitive reduction under interval‑finiteness]
If \((X,\preceq)\) is interval‑finite then the cover relation \(C\)
exists, the transitive reduction exists, is unique, and equals \(C\).
\end{theorem}

\begin{proof}
Same as finite case (Luckraz-Pansera (2019)) applied to each finite interval \([x,y]\).
\end{proof}

\section{Symmetric part, equivalence classes, and quotient order}\label{sec:quotient}
\subsection{Symmetric part allowing cycles of length one and two}
\begin{definition}[Symmetric generator]
Define the \emph{symmetric generator} relation \(S\subseteq V\times V\) by

\[
(u,v)\in S \quad\Longleftrightarrow\quad (u=v)\ \text{or}\ \big((u,v)\in E\ \text{and}\ (v,u)\in E\big).
\]

Thus \(S\) collects self‑loops and symmetric pairs (two‑cycles).
\end{definition}

\begin{definition}[Equivalence relation generated by cycles of length one and two]
Let \(\sim_S\) be the equivalence relation on \(V\) generated by \(S\):
the smallest equivalence relation containing \(S\). Concretely, \(u\sim_S v\)
iff there exists a finite sequence \(u=w_0,w_1,\dots,w_k=v\) with each
\((w_{j-1},w_j)\in S\).
\end{definition}

\begin{remark}
This equivalence groups vertices that are connected by chains of
self‑loops and symmetric pairs. It is strictly weaker than collapsing
entire strongly connected components (SCCs) when longer cycles exist;
we intentionally focus on cycles of length one and two because these are
the symmetric anomalies we allow while still treating longer directed
cycles as genuine obstructions to time structure unless they are
collapsed by further modeling choices.
\end{remark}

\subsection{Quotient set and induced relation}
\begin{definition}[Quotient set and induced reachability]
Let \(X:=V/\!\sim_S\) be the set of \(\sim_S\)-equivalence classes. For
classes \(x,y\in X\) define \(x\preceq y\) iff there exist representatives
\(u\in x\), \(v\in y\) with \(uR_G v\) (reachability in the raw graph).
Write \(x\prec y\) when \(x\preceq y\) and \(x\neq y\).
\end{definition}

\begin{proposition}[Well‑definedness]
The relation \(\preceq\) is well defined and antisymmetric on \(X\):
if \(x\preceq y\) and \(y\preceq x\) then \(x=y\).
\end{proposition}

\begin{proof}
Well‑definedness: if \(u,u'\in x\) and \(v,v'\in y\) and \(uR v\), then
because \(u\sim_S u'\) and \(v\sim_S v'\) and \(\sim_S\) is generated by
edges that are symmetric or reflexive, reachability from \(u'\) to \(v'\)
still holds (concatenate short symmetric steps). Antisymmetry: if
\(x\preceq y\) and \(y\preceq x\) then there exist representatives with
mutual reachability; by construction of \(\sim_S\) this implies the
classes coincide.
\end{proof}

\begin{remark}
We call \((X,\preceq)\) the \emph{reduced quotient poset}. All subsequent
ordering, transitive reduction, and time labeling statements refer to
\((X,\preceq)\).
\end{remark}

\section{Finite case: characterization}\label{sec:finite}
\begin{theorem}[Finite corrected characterization]
Let \(G=(V,E)\) be finite. Form the equivalence \(\sim_S\) generated by
self‑loops and symmetric pairs and the quotient poset \((X,\preceq)\).
Then \(G\) admits a time structure (a real labeling strictly increasing
along edges) iff the quotient poset \((X,\preceq)\) is acyclic and its
transitive reduction is unique. In that case the transitive reduction of
\((X,\preceq)\) equals the cover relation and lifts to a minimal edge set
on \(G\) (unique up to edges internal to \(\sim_S\)-classes).
\end{theorem}

\begin{proof}
(\(\Rightarrow\)) If \(G\) admits a time labeling \(f:V\to\mathbb R\)
with \(u\prec v\Rightarrow f(u)<f(v)\), then \(f\) is constant on each
\(\sim_S\)-class (because symmetric edges force equality) and induces a
strict labeling on \(X\). Hence \(X\) is acyclic and its transitive
reduction is unique (cover relation).

(\(\Leftarrow\)) If \((X,\preceq)\) is acyclic and has a unique
transitive reduction \(C\), choose a topological ordering of \(X\) and
assign strictly increasing real values to classes. Lift to \(V\) by
setting each vertex's label equal to its class label. For any edge
\((u,v)\in E\) with \(u\) in class \(x\) and \(v\) in class \(y\), we
have \(x\preceq y\) and hence the lifted labels satisfy \(f(u)\le f(v)\),
with strict inequality whenever \(x\prec y\). Because symmetric edges
were collapsed, symmetric edges lie inside classes and do not
require strict inequality. Thus the lifted labeling is a valid time
structure for \(G\) in the weak sense (strict across class‑level edges).
Uniqueness and lifting arguments follow as in the finite classical case.
\end{proof}

\begin{remark}
This theorem permits reflexive and symmetric edges while preserving
the meaningful notion of a time structure after collapsing the symmetric
part. It also clarifies that uniqueness of the transitive reduction must
be tested on the quotient poset \(X\).
\end{remark}

\section{Countable case: existence, uniqueness and algorithms}\label{sec:countable_alg}
\subsection{Maximal chains and Zorn argument on the quotient}
\begin{lemma}[Maximal chains on the quotient]
Let \(G=(V,E)\) be countable and let \(X=V/\!\sim_S\) be the quotient.
If \(x\preceq y\) in \(X\) then there exists a maximal finite chain of
classes from \(x\) to \(y\) (maximal with respect to insertion of
intermediate classes).
\end{lemma}

\begin{proof}
Apply the subsequence/Zorn argument to the set of finite class chains
from \(x\) to \(y\). Acyclicity of \(X\) (if assumed) prevents repetition
of classes and ensures chains are finite; Zorn's Lemma yields a maximal
chain.
\end{proof}

\subsection{Algorithmic construction}
When \(V\) and \(E\) are recursively enumerable, implement the following
procedure:

\begin{enumerate}
  \item Enumerate vertices and edges; maintain a dynamic partition into
    discovered \(\sim_S\)-classes (merge vertices when symmetric edges or
    self‑loops are observed).
  \item Maintain the quotient graph on classes and compute
    reachability within the discovered quotient.
  \item For each discovered class‑edge \((x,y)\) include it in the
    candidate reduction unless a  intermediate class \(z\)
    satisfies \(xR z\) and \(zR y\).
  \item Continue until stabilization: if the candidate set remains
    unchanged for sufficiently many enumeration stages and no new classes
    appear for a long prefix, output the candidate.
\end{enumerate}

\begin{proposition}[Verifiability and termination under finiteness]
If the quotient \(X\) is finite then the algorithm terminates and returns
the unique transitive reduction. If \(X\) is countable and the true
transitive reduction is finite and recursively enumerable, the algorithm
stabilizes and returns it.
\end{proposition}

\begin{proof}
Standard enumeration and stabilization arguments applied to the quotient
graph; correctness follows because redundancy decisions are class‑level
and determined once all relevant intermediate classes are discovered.
\end{proof}

\section{Acyclicity and unique transitive reduction: generalized characterization}
\label{sec:acyclicity-uniqueTR}

In the finite setting Luckraz \& Pansera (2023)
established the close connection between acyclicity of the informational
digraph and uniqueness of its transitive reduction. In the previous sections, we gave a new result
that collapses reflexive and symmetric anomalies (cycles of length one
and two) into equivalence classes, the natural statement is about the
quotient poset \(X=\mathcal I/\!\sim_S\). The results below give a
single, unified characterization that (i) recovers the finite theorem,
(ii) extends to countable posets under natural finiteness hypotheses,
and (iii) clarifies the measurable/continuum cases by stating the exact
topological hypotheses needed for existence and uniqueness.

Throughout this section let \(G=(\mathcal I,E)\) be the raw informational
digraph, let \(\sim_S\) be the equivalence relation generated by
self‑loops and symmetric pairs (cycles of length one and two), and let
\(X:=\mathcal I/\!\sim_S\) denote the quotient set. Denote the induced
partial order on \(X\) by \(\preceq\) and its strict part by \(\prec\).
Write \(C\) for the cover relation on \(X\) when it exists.

\subsection{Finite and countable posets}

\begin{theorem}[Generalized finite/countable characterization]\label{thm:generalized-finite-countable}
Let \((X,\preceq)\) be a poset obtained by collapsing the symmetric part
\(\sim_S\) of an informational digraph. Then:
\begin{enumerate}[leftmargin=*]
  \item \textbf{(Necessity)} If there exists a transitive reduction \(R\)
    of \((X,\prec)\) and \(R\) is unique, then \((X,\preceq)\) is acyclic.
  \item \textbf{(Sufficiency under interval‑finiteness)} If
    \((X,\preceq)\) is acyclic and interval‑finite (every interval
    \([x,y]\) is finite), then the cover relation \(C\) exists, the
    transitive reduction exists, is unique, and equals \(C\).
  \item \textbf{(Finite case)} If \(X\) is finite then (2) applies and
    the cover relation \(C\) is the unique transitive reduction.
\end{enumerate}
\end{theorem}

\begin{proof}
\textbf{(1) Necessity.} Suppose a unique transitive reduction \(R\) of
\((X,\prec)\) exists. If \((X,\preceq)\) contained a directed cycle
\(x_1\prec x_2\prec\cdots\prec x_k\prec x_1\) with \(k\ge 1\), then no
finite edge set could be minimal in the sense of generating \(\prec\)
because along the cycle edges can be bypassed by traversing the cycle,
and one can construct distinct minimal generating sets by choosing which
cycle edges to keep. Concretely, along a directed cycle every edge is
reachable via the remaining cycle edges, so minimality fails or is not
unique. Therefore uniqueness of \(R\) implies there are no directed
cycles in \((X,\prec)\); i.e., \((X,\preceq)\) is acyclic.

\textbf{(2) Sufficiency under interval‑finiteness.} Assume \((X,\preceq)\)
is acyclic and interval‑finite. Fix \(x\prec y\). The interval
\([x,y]=\{z\in X: x\preceq z\preceq y\}\) is finite by hypothesis. Choose
a chain \(x=z_0\prec z_1\prec\cdots\prec z_m=y\) of minimal length \(m\).
By minimality each consecutive pair \(z_i\prec z_{i+1}\) is a cover;
otherwise we could insert an intermediate element and shorten the chain.
Collecting all such cover pairs yields the cover relation \(C\). For any
\(x\prec y\) the chain of covers from \(x\) to \(y\) shows the transitive
closure of \(C\) equals \(\prec\). Minimality of \(C\) follows because
omitting any cover pair \((z_i,z_{i+1})\) would break reachability
between \(z_i\) and \(z_{i+1}\) (there is no intermediate class to
reestablish it). Hence \(C\) is a transitive reduction.

Uniqueness: suppose \(F\) is any other edge set whose transitive closure
equals \(\prec\). For any cover pair \((z_i,z_{i+1})\in C\), if
\((z_i,z_{i+1})\notin F\) then there must be a path from \(z_i\) to
\(z_{i+1}\) using edges in \(F\), which would pass through some
intermediate \(z\) with \(z_i\prec z\prec z_{i+1}\), contradicting that
\((z_i,z_{i+1})\) is a cover. Thus \(C\subseteq F\). Therefore \(C\) is
contained in every generating set and is the unique minimal generating
set, i.e. the unique transitive reduction.

\textbf{(3) Finite case.} If \(X\) is finite then every interval is finite,
so interval‑finiteness holds and (2) applies. This recovers the classical
finite uniqueness result.
\end{proof}

\begin{corollary}[Equivalence for practical modeling]
For an informational digraph \(G\) with quotient poset \(X\) obtained by
collapsing \(\sim_S\):

\[
\text{\(X\) is acyclic and interval‑finite}
\quad\Longleftrightarrow\quad
\text{the transitive reduction of \(X\) exists and is unique (equals \(C\)).}
\]

Consequently, under these hypotheses \(G\) admits a (weak) time structure
obtained by assigning strictly increasing real labels to classes in \(X\)
and lifting them to \(\mathcal I\).
\end{corollary}

\subsection{Necessity of finiteness hypotheses}
The interval‑finiteness hypothesis in Theorem \ref{thm:generalized-finite-countable}
is essential for existence of a transitive reduction in the infinite
case. There are acyclic posets (for example, dense orders like
\(\mathbb Q\cap[0,1]\)) that admit order embeddings into \(\mathbb R\)
(i.e., time labelings) but have no covers and hence no transitive
reduction in the sense of a minimal generating edge set. Thus:
\begin{itemize}
  \item \emph{Acyclicity alone} guarantees the existence of a time
    labeling (topological ordering) in the countable case (via topological
    sorting) but does not guarantee the existence of a minimal finite
    generating edge set unless intervals are finite.
  \item \emph{Uniqueness of transitive reduction} (as a combinatorial
    object) is equivalent to acyclicity plus the structural condition
    that every ordered pair is generated by finitely many intermediate
    classes (interval‑finiteness).
\end{itemize}

\subsection{Interpretation: redundant information sets}
\begin{definition}[Redundant information]
An information set \(I\) is \emph{redundant} if removing some outgoing
edge from \(I\) does not change the reachability relation \(R_G\). The
graph \(G\) \emph{contains redundant information} if its transitive
reduction is not unique.
\end{definition}

\begin{corollary}
In the finite case \(G\) contains no redundant information iff the
transitive reduction is unique iff the cover relation on the quotient
poset \(X\) equals the transitive reduction.
\end{corollary}

\begin{remark}
This restates and clarifies the finite characterization: the quotient
poset \(X\) is the canonical object to be ordered by any clock; uniqueness
of the transitive reduction on \(X\) is equivalent to absence of
redundant information.
\end{remark}

\section{Continuum rooted real trees}\label{sec:continuum_tree}
We now move to continuum trees built from continuous height functions and
show how to place Kuhn games on them.

\subsection{Height function and pseudo‑metric}
Let \(g:[0,1]\to\mathbb R_{\ge0}\) be continuous with \(g(0)=g(1)=0\).
Define

\[
m_g(s,t):=\min_{u\in[s\wedge t,s\vee t]} g(u),\qquad
d_g(s,t):=g(s)+g(t)-2m_g(s,t).
\]

\subsection{Quotient tree \(T_g\)}
Define \(s\sim t\) iff \(d_g(s,t)=0\). The quotient
\(T_g:=[0,1]/\!\sim\) equipped with the metric induced by \(d_g\) is a
compact rooted real tree. Denote the equivalence class of \(s\) by \([s]\).
The root \(r\) is the class of a global minimum of \(g\) (take \(r=[0]\)).

\begin{remark}
\(T_g\) is a compact metric real tree: any two points are joined by a
unique arc isometric to a compact interval. The metric \(d_g\) is
continuous and \(T_g\) is Polish.
\end{remark}

 [Prefix order on \(T_g\)]
Define a strict partial order \(\prec\) on \(\mathcal H:=T_g\) by:

\[
[s]\prec[t]\quad\Longleftrightarrow\quad [s]\neq[t]\ \text{and the geodesic from }r\text{ to }[t]\text{ passes through }[s].
\]

A geodesic between two points \(x,y\in T_g\) is the unique shortest path joining them; formally it is an isometry

\[
\gamma:[0,d_g(x,y)]\to T_g,\qquad \gamma(0)=x,\ \gamma(d_g(x,y))=y.
\]

In \(T_g\) every pair of points is joined by a unique geodesic.

\begin{example} [Linear / constant function (degenerate tree)]
Function: The function is \(g(t)=0\) for every \(t\in[0,1]\).

What the quotient does: For every \(s,t\) we have \(m_g(s,t)=0\) and \(d_g(s,t)=0\). Thus all points of \([0,1]\) are identified and \(T_g=\{[0]\}\), a single point (the root).
Geodesics:There is only one point, so geodesics are trivial. This is the degenerate tree with no edges and no branching.
\end{example}

\begin{example}[Interior maximum (single peak, no branching]
Piecewise linear single peak (written without a cases environment):
Define \(g\) by

\[
g(0)=0,\qquad g(1)=0,
\]

and for \(t\in[0,1]\) set

\[
\text{for }0\le t\le \tfrac{1}{2}\text{, } g(t)=2t;\qquad
\text{for }\tfrac{1}{2}\le t\le 1\text{, } g(t)=2(1-t).
\]

Then \(g(\tfrac{1}{2})=1\) is the unique interior maximum.

How the quotient looks:
Because \(g\) has no interior local minima (only a single global maximum), the quotient \(T_g\) is isometric to a single compact interval. Concretely \(T_g\) is a line segment: distances between the two extreme leaves equal the total up‑and‑down height encoded by \(g\).

Geodesics made explicit:
If \(x=[s]\) and \(y=[t]\) with \(s<t\), the geodesic \([x,y]\) is the image in \(T_g\) of the interval \([s,t]\) (after collapsing identified points). It is isometric to \([0,d_g(x,y)]\). There are no branchpoints: every interior point of the tree has degree \(2\).
\end{example}
\begin{example} [Interior local minimum between two peaks (cutpoint but not a branchpoint)]
Piecewise linear two‑peak with interior dip (written without a cases environment):
Define \(g\) by specifying values on subintervals:

\[
\text{for }0\le t\le \tfrac{1}{4}\text{, } g(t)=4t;
\qquad
\text{for }\tfrac{1}{4}\le t\le \tfrac{3}{4}\text{, } g(t)=1-2\bigl|t-\tfrac{1}{2}\bigr|;
\]

\[
\text{for }\tfrac{3}{4}\le t\le 1\text{, } g(t)=4(1-t).
\]

Graphically: rise from \(0\) to a first peak, drop to a local minimum near \(t=\tfrac{1}{2}\), then rise again and finally fall to \(0\) at \(1\).

What happens in \(T_g\):
The local minimum at \(t_0\approx\tfrac{1}{2}\) becomes a cutpoint: removing the class \([t_0]\) separates the tree into two connected components (left and right). This point is not a branchpoint of degree \(\ge 3\); it has degree \(2\) if exactly two excursions meet there. The tree looks like two arcs joined at a single interior point.

Geodesics:
A geodesic between a point on the left arm and a point on the right arm passes through \([t_0]\). The geodesic is the concatenation of the left subarc and the right subarc meeting at \([t_0]\). Uniqueness of the geodesic still holds.
\end{example}
\begin{example} [True branching from multiple excursions (branchpoint of degree 3)]
Three disjoint excursions meeting at the same zero level"
Partition \([0,1]\) into three subintervals \(I_1=[0,\tfrac{1}{3}],\ I_2=[\tfrac{1}{3},\tfrac{2}{3}],\ I_3=[\tfrac{2}{3},1]\). On each \(I_k\) define a triangular excursion of height \(1\) that starts and ends at \(0\). For example, on \(I_1\) set

\[
g(t)=6\min\!\bigl(t,\tfrac{1}{3}-t\bigr),\qquad t\in\bigl[0,\tfrac{1}{3}\bigr].
\]

and define \(g\) similarly on \(I_2\) and \(I_3\).

What the quotient looks like:
All zeros at the boundaries of the three excursions are identified in the quotient to a single class \([0]\) (the global minimum). The class \([0]\) is a branchpoint of degree \(3\): removing \([0]\) yields three connected components (the three excursion branches). The tree \(T_g\) is a star with three arms (each arm is an interval of length equal to the excursion height).

Geodesics:
A geodesic between two points on the same arm is the subinterval of that arm. A geodesic between points on different arms goes from the first point to the root \([0]\) then out to the second point; it is the concatenation of two arm segments and is unique.
\end{example}

\section{Extensive games on continuum trees}\label{sec:kuhn_continuum}
We place an extensive form on the continuum tree.

\subsection{Decision vertices and actions}
Let \(\mathcal H^D\subseteq\mathcal H\) be the set of decision vertices,
assumed Borel. Let \(N=\{1,\dots,n\}\) be the finite player set and
\(P:\mathcal H^D\to N\) a Borel player assignment.

\paragraph{Global histories.}
Let \(\mathcal H\) be the set of (finite or infinite) histories that start at the root and carry the usual prefix order \(\preceq\).

\paragraph{Extensions of a node.}
For \(h\in\mathcal H\) write

\[
\operatorname{Ext}(h)\;=\;\{\,v\in\mathcal H:\;h\preceq v\,\},
\]

the set of global histories that extend \(h\).

\paragraph{Histories rooted at \(h\) (continuations).}
For each decision node \(h\in\mathcal H^D\) define the set of continuations

\[
\mathcal H_h \;=\; \{\,a:\; a\text{ is a finite or infinite sequence of moves starting at }h\,\}.
\]

There is a canonical bijection \(\phi_h:\operatorname{Ext}(h)\to\mathcal H_h\) given by taking the suffix of a global history after \(h\). Under this bijection the empty continuation corresponds to the global history \(h\) itself.

\paragraph{Action sets.}
For each decision node \(h\) take the action set \(\mathcal A_h\subseteq\mathcal H_h\). Thus an action \(a\in\mathcal A_h\) is a continuation (a tail) starting at \(h\), not a global history from the root.

\paragraph{Measurability requirement}
In either formulation require that each \(\mathcal A_h\) is a standard Borel space and that the map or correspondence \((h,a)\mapsto\tau(h,a)\) has a Borel graph, so measurable selection and probability constructions apply.

\subsection{Information partition}
For each player \(i\) let \(\mathcal I_i\) be a partition of \(P_i\) into
disjoint nonempty Borel cells. Elements \(I\in\mathcal I_i\) are
information sets. Assume action consistency and no self‑ancestry.

\subsection{Informational digraph on continuum trees}
Define \(G=(\mathcal I,E)\) with

\[
(I',I'')\in E \quad\Longleftrightarrow\quad \exists\,h'\in I',\ \exists\,h''\in I'':\ h'\prec h''.
\]

The quotient poset \(X=\mathcal I/\!\sim\) is defined as before.

\section{Time structure in the continuum}\label{sec:continuum_time}
Continuum representations require explicit topological hypotheses. We
state a Debreu‑style theorem and measurable variants.

\subsection{Order‑dense separability}
\begin{definition}[Order‑dense countable subset]
A countable set \(D\subset V\) is \emph{order‑dense} for \(\prec\) if for
every \(x\prec y\) there exists \(d\in D\) with \(x\prec d\prec y\).
\end{definition}

\subsection{Debreu‑style continuous representation}
\begin{theorem}[Continuous time labeling for acyclic continuum trees]\label{thm:debure_continuum}
Let \(V\) be a second‑countable metric space and \(\prec\subseteq V\times V\)
a strict partial order satisfying:
\begin{enumerate}[leftmargin=*]
  \item \(\prec\) is acyclic,
  \item \(\prec\) is topologically closed in \(V\times V\),
  \item there exists a countable order‑dense subset \(D\subset V\).
\end{enumerate}
Then there exists a continuous function \(f:V\to\mathbb R\) such that
\(x\prec y\Rightarrow f(x)<f(y)\).
\end{theorem}

\begin{proof}[Proof Sketch - see the appendix for the full proof]
Adapt Debreu's separability construction. Use \(D=\{d_1,d_2,\dots\}\) to
define open separating neighborhoods \(L_i=\{x:x\prec d_i\}\) and
\(R_i=\{x:d_i\prec x\}\). Closedness of \(\prec\) ensures these sets are
open. Recursively construct a nested family of open sets \(U_q\) indexed
by rationals so that for every \(x\prec y\) some \(U_q\) separates them.
Define \(f(x)=\inf\{q:x\in U_q\}\). Standard arguments using second
countability and order‑density yield continuity and strict monotonicity.
\end{proof}

\subsection{Borel measurable labeling under weaker hypotheses}
\begin{theorem}[Borel measurable time labeling]\label{thm:polish}
Let \(G=(V,E)\) be a topological informational digraph with \(V\) Polish
and \(E\) Borel. Suppose the quotient of \(V\) by strongly connected
components (SCC)	 is acyclic and the quotient is standard Borel. Then there
exists a Borel measurable \(f:V\to\mathbb R\) with \((u,v)\in E\Rightarrow f(u)<f(v)\).
\end{theorem}

\begin{proof}[Proof Sketch - see the appendix for the full proof]
Collapse SCCs to obtain a standard Borel quotient \(\widetilde V\) with a
Borel acyclic partial order. Use measurable selection (Kuratowski–Ryll‑Nardzewski)
to choose minimal elements in a Borel way and enumerate them; assign
rational labels in increasing order to obtain a Borel linear extension,
then lift to \(V\).
\end{proof}

\begin{remark}
Interval‑finiteness and uniqueness in continuum
If the quotient poset \(X\) is interval‑finite (finite intervals along
every pair) then the cover relation exists and yields a unique transitive
reduction; this is a strong structural condition that often fails in
continuum models with accumulation points.

\end{remark}

\section{Examples and counterexamples}\label{sec:examples}
\subsection{Piecewise linear continuum tree}
Let \(g\) be piecewise linear with peaks at \(0.25,0.75\) and valley at
\(0.5\) as in the running example. The quotient tree \(T_g\) contains
classes \(A=[0.25],V=[0.5],B=[0.75]\). Selecting \(N=\{A,V,B\}\) yields a
finite path \(A-V-B\) and (H) applies.

\subsection{Dense order without covers}
Let \(X\) be order‑isomorphic to \(\mathbb Q\cap[0,1]\). There are no
covers and no transitive reduction; a time labeling exists (embedding into
\(\mathbb R\)) but no minimal generating edge set exists.

\subsection{Zeno accumulation}
Let classes at \(t_n=1-2^{-n}\) and a limit class at \(t_\infty=1\).
Information sets distinguishing each \(t_n\) produce accumulation and
may violate interval‑finiteness; discrete ranks fail and continuous
height may assign identical limit values to many classes.

\section{Differential‑game application}\label{sec:differential}
We illustrate the continuum construction with a differential game (see Ibragimov, Tursunaliev and Luckraz (2024) for example).

\subsection{Model}
State \(x(t)\in\mathbb R\) on \(t\in[0,1]\) evolves by

\[
\dot x(t)=u(t)+v(t),\qquad x(0)=x_0,
\]

with controls \(u\in\mathcal U\), \(v\in\mathcal V\) measurable in
compact sets \(U,V\). Observations: players observe \(y(t)=x(t)+\eta(t)\)
with \(\eta\) a continuous noise path sampled at the start. Decision
times \(\mathcal T\subset[0,1]\) determine cut times.

\subsection{Continuum tree and information sets}
Paths \(\gamma\in C([0,1],\mathbb R)\) form \(\mathcal P\). Fix cut times
\(0<t_1<\dots<t_K<1\). Information sets at \(t_k\) are Borel subsets of
\(\pi_{t_k}(\mathcal P)\). The continuum tree \(T_g\) with \(g(s)=s\)
identifies histories with time coordinates; (H) holds and \(f(I)=\inf_{h\in I}h\)
gives the natural clock.

\subsection{Accumulation example}
If \(\mathcal T\) accumulates at \(t_\infty\), and information sets
distinguish arbitrarily fine observations near \(t_\infty\), interval
finiteness fails. Remedies: coarsen information sets at the accumulation
point or model decisions via filtrations and stopping times.

\section{Conclusion}

In this paper, we analyzed time structure in extensive games across finite, countable, and continuum informational digraphs. By lifting the quotient/equivalence viewpoint we have unified finite characterizations (Luckraz and Pansera 2023) with continuum representation theory and shown precisely which structural and topological hypotheses are required to carry finite intuitions into the continuum. Our existence and uniqueness results for transitive reductions under interval‑finiteness, together with the necessary and sufficient conditions for real‑valued time labelings, can apply when assessing whether a proposed extensive‑form model admits a coherent global clock.

Beyond these foundational results, the paper solves the time structure problem for a continuum tree by using  measurable selection constructions in Polish spaces that produce Borel time labelings when the topological hypotheses hold. These contributions make the theory appealing  for applied researchers working on dynamic games, differential games, and mechanism design in continuous time, and they clarify connections to computational and equilibrium issues raised in the literature (von Stengel and Forges 2008; Hansen, Ibsen, and Neyman 2021). The differential game example demonstrates how continuum trees and quotient posets arise in natural economic and control problems and how our methods resolve representational ambiguities that would otherwise impede analysis.

Our work also points to several directions for future work. First, relaxing or replacing the sigma compactness and order density hypotheses in specific application domains could broaden applicability while preserving measurable constructions. Second, a systematic complexity analysis of the algorithms we provide would clarify their practical limits and suggest optimized implementations. Third, stochastic and dynamic extensions, where information sets evolve randomly or depend on continuous signals, would test the robustness of the quotient approach in probabilistic environments. Finally, exploring implications for equilibrium existence and refinement in continuum informational settings may yield new insights into how time structure interacts with strategic behavior, complementing classical sequence‑based and graph‑based equivalences (Kline and Luckraz 2016) and enriching the toolkit available to theorists and practitioners working at the interface of game theory, topology, and computation.

\section{Appendix}
This appendix collects full, self‑contained proofs of the main technical
results stated in the paper. We repeat minimal notation for convenience.

\subsection*{Notation and preliminaries}

\begin{itemize}
  \item $G=(V,E)$ denotes a directed graph (possibly finite, countable, or
    with a topological vertex set).
  \item For $u,v\in V$, $uR_G v$ (or simply $uR v$) means $v$ is reachable
    from $u$ via a finite directed walk in $G$.
  \item $\sim$ denotes mutual reachability: $u\sim v$ iff $uR v$ and $vR u$.
  \item $X=V/\!\sim$ denotes the quotient set of equivalence classes; we
    write $[u]$ for the class of $u$.
  \item For a partial order $(X,\preceq)$, $x\prec y$ means $x\preceq y$ and
    $x\neq y$. The interval $[x,y]=\{z\in X: x\preceq z\preceq y\}$.
  \item A \emph{transitive reduction} of a reachability relation is a
    directed graph whose transitive closure equals the reachability
    relation and whose edge set is minimal with that property.
\end{itemize}

\subsection*{Finite case: uniqueness of transitive reduction}
\begin{theorem}[Finite uniqueness]
Let $G=(V,E)$ be a finite directed graph. Collapse mutual reachability
classes to obtain the quotient poset $(X,\preceq)$. Then the cover relation
$C$ on $X$ exists, the transitive reduction of $(X,\prec)$ exists, is
unique, and equals $C$. Lifting $C$ to $V$ (choosing representatives)
yields the unique transitive reduction of $G$ up to edges internal to
strongly connected components.
\end{theorem}

\begin{proof}
Because $V$ is finite, $X$ is finite. The relation $\preceq$ on $X$ is a
partial order: reflexivity and transitivity follow from reachability, and
antisymmetry holds on classes by construction.

Fix $x\prec y$ in $X$. The interval $[x,y]$ is finite, so there exists a
chain $x=z_0\prec z_1\prec\cdots\prec z_m=y$ of minimal length $m\ge1$.
By minimality each consecutive pair $z_k\prec z_{k+1}$ is a cover: if
there were $z$ with $z_k\prec z\prec z_{k+1}$ we could shorten the chain.
Thus every ordered pair is reachable via a finite sequence of cover edges,
so the transitive closure of $C$ equals $\prec$.

Minimality: suppose $F\subseteq X\times X$ is any edge set whose transitive
closure equals $\prec$. For any cover pair $(z_k,z_{k+1})\in C$, if
$(z_k,z_{k+1})\notin F$ then there must be a path from $z_k$ to $z_{k+1}$
using edges in $F$, which would pass through some intermediate $z$ with
$z_k\prec z\prec z_{k+1}$, contradicting that $(z_k,z_{k+1})$ is a cover.
Hence every cover pair must belong to $F$, so $C\subseteq F$. Therefore
$C$ is contained in every generating edge set and is minimal; uniqueness
follows.

Lifting to $V$: for each cover pair $([u],[v])\in C$ choose one edge
$(u',v')\in E$ with $u'\in[u], v'\in[v]$ that witnesses the class
precedence (such an edge exists by definition of the quotient). The set of
these chosen edges, together with arbitrary choices inside SCCs if needed,
yields an edge set on $V$ whose transitive closure equals $R_G$ and which
is minimal up to edges internal to SCCs. This completes the proof.
\end{proof}

\subsection*{Countable case: maximal walks and unique transitive reduction}
\subsubsection{Existence of maximal walks}
\begin{lemma}[Existence of maximal walks]
Let $G=(V,E)$ be an acyclic directed graph (countable). If there is a walk
from $u$ to $v$, then there exists a maximal walk from $u$ to $v$ (maximal
with respect to subsequence inclusion).
\end{lemma}

\begin{proof}
Let $\Theta$ be the set of all finite walks from $u$ to $v$. Define a
partial order on $\Theta$ by $\alpha\preceq\beta$ iff $\alpha$ is a
subsequence of $\beta$ (i.e., $\alpha$ can be obtained by deleting some
vertices from $\beta$ while preserving order). We will apply Zorn's Lemma
to $(\Theta,\preceq)$.

Take any chain $\mathcal C\subseteq\Theta$ totally ordered by $\preceq$.
We claim $\mathcal C$ has an upper bound in $\Theta$. Because $\mathcal C$
is totally ordered by subsequence, for any two walks $\alpha,\beta\in\mathcal C$
either $\alpha$ is a subsequence of $\beta$ or vice versa. Enumerate the
walks in $\mathcal C$ as $\alpha_1\preceq\alpha_2\preceq\cdots$ (finite or
countable). Construct a finite sequence $\gamma$ by listing the vertices
in the order they first appear in the chain: start with the vertices of
$\alpha_1$, then append the new vertices that appear in $\alpha_2$ but
not in $\alpha_1$, and so on. Because each $\alpha_i$ is finite and
acyclicity prevents repetition along a walk, this process yields a finite
walk $\gamma$ from $u$ to $v$ that contains every $\alpha\in\mathcal C$
as a subsequence. Thus $\gamma$ is an upper bound of $\mathcal C$.

By Zorn's Lemma, $\Theta$ has a maximal element with respect to $\preceq$,
i.e., a walk that is not a proper subsequence of any other walk from $u$
to $v$. This is the desired maximal walk.
\end{proof}

\subsection*{Uniqueness of transitive reduction in acyclic countable graphs}
\begin{lemma}[Unique transitive reduction]
Let $G=(V,E)$ be an acyclic directed graph (countable). Define

\[
F:=\{(u,v)\in E:\ \nexists\, w\in V\setminus\{u,v\}\ \text{with } uR w \text{ and } wR v\}.
\]

Then $H=(V,F)$ is the unique transitive reduction of $G$.
\end{lemma}

\begin{proof}
We show three properties: (i) $R_H=R_G$, (ii) $F$ is minimal, and (iii)
$F$ is unique.

(i) Since $F\subseteq E$, $R_H\subseteq R_G$ is immediate. Conversely,
suppose $uR_G v$. By the previous lemma there exists a maximal walk
$u=I_1,\dots,I_k=v$. For each consecutive pair $(I_j,I_{j+1})$ in this
maximal walk, there cannot exist any intermediate vertex $w$ with
$I_jR w$ and $wR I_{j+1}$; otherwise the walk would not be maximal. Thus
each consecutive pair belongs to $F$. Hence $uR_H v$. Therefore $R_H=R_G$.

(ii) Minimality: if we remove any edge $(u,v)\in F$ from $H$, then by
definition there is no $w\neq u,v$ with $uR w$ and $wR v$. Thus removing
$(u,v)$ breaks reachability from $u$ to $v$ in the remaining graph, so
the transitive closure no longer equals $R_G$. Therefore $F$ is minimal.

(iii) Uniqueness: suppose $F'$ is another minimal edge set with transitive
closure equal to $R_G$. If $F'\neq F$ then there exists $(u,v)\in F\setminus F'$.
But since $F'$ generates $R_G$, there must be a path from $u$ to $v$ using
edges in $F'$, which implies the existence of some intermediate $w$ with
$uR w$ and $wR v$, contradicting the definition of $(u,v)\in F$. Hence
$F'=F$.

This proves uniqueness.
\end{proof}

\subsection*{Interval‑finite posets and transitive reductions}
\begin{theorem}
Let $(X,\preceq)$ be a poset. If $(X,\preceq)$ is interval‑finite (every
interval $[x,y]$ is finite) then the cover relation $C$ exists, the
transitive reduction exists, is unique, and equals $C$.
\end{theorem}

\begin{proof}
Fix $x\prec y$. Because $[x,y]$ is finite there exists a chain
$x=z_0\prec z_1\prec\cdots\prec z_m=y$ of minimal length $m$. By
minimality each consecutive pair $z_k\prec z_{k+1}$ is a cover; otherwise
we could insert an intermediate element and shorten the chain. Thus every
$x\prec y$ is reachable by a finite sequence of cover edges, so the
transitive closure of $C$ equals $\prec$. Minimality and uniqueness
follow by the same argument as in the finite case: omitting any cover edge
would break reachability between its endpoints, so $C$ is contained in
every generating edge set and hence is the unique minimal generating set.
\end{proof}

\subsection*{Topological sorting for countable acyclic digraphs}
\begin{theorem}
Let $G=(V,E)$ be an acyclic directed graph with $V$ countable. Then there
exists $f:V\to\mathbb Q$ such that $(u,v)\in E\Rightarrow f(u)<f(v)$.
\end{theorem}

\begin{proof}
Enumerate $V=\{v_1,v_2,\dots\}$. We construct a topological ordering by
iteratively removing vertices of in‑degree zero from the remaining graph.

At stage $k$ consider the finite induced subgraph on the first $N_k$
vertices not yet removed (choose $N_k$ large enough to include some
unremoved vertices). Because the induced subgraph is finite and acyclic,
it has a vertex of in‑degree zero. Choose such a vertex $w_k$, append it
to the ordering, and remove it from the remaining graph. Continue this
process; every vertex will eventually be removed because each step
removes one vertex and the graph is countable. Assign rational labels
$f(w_k)=k$ (or any strictly increasing rational sequence). Then for any
edge $(u,v)\in E$, $u$ must appear earlier than $v$ in the ordering, so
$f(u)<f(v)$.
\end{proof}
\subsection*{Counter-examples}
In this subsection we provide some counter-examples of games where a time structure may not exist.
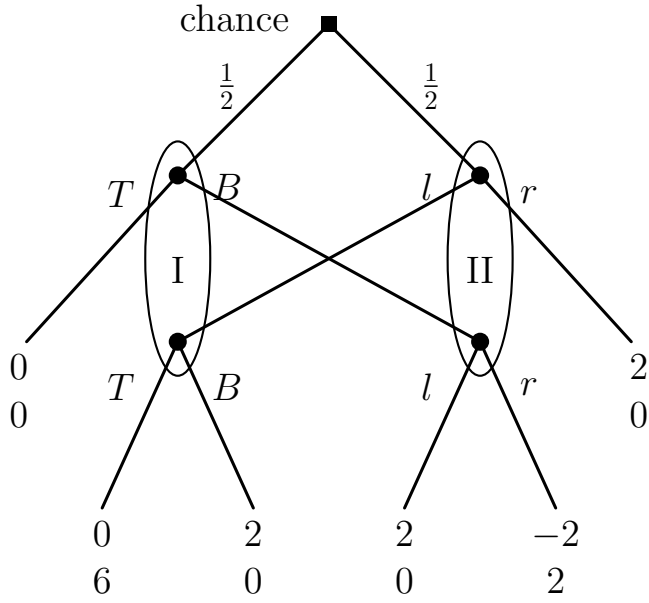
\begin{figure}[htbp]
\begin{tikzpicture}[
    x=1cm,
    y=1cm,
    branch/.style={
        draw=black,
        line width=1.15pt,
        line cap=round
    },
    nodecircle/.style={
        circle,
        fill=black,
        inner sep=2.4pt
    },
    chance/.style={
        rectangle,
        fill=black,
        inner sep=3.0pt
    },
    infoset/.style={
        draw=black,
        line width=0.8pt
    },
    lab/.style={
        font=\Large\itshape,
        inner sep=1pt
    },
    num/.style={
        font=\Large,
        inner sep=1pt
    },
    payoff/.style={
        font=\Large,
        align=center,
        inner sep=1pt
    }
]


\coordinate (C)  at (0,5.0);

\coordinate (LU) at (-2.0,3.0);
\coordinate (RU) at (2.0,3.0);

\coordinate (LL) at (-2.0,0.8);
\coordinate (RL) at (2.0,0.8);

\coordinate (TL)  at (-4.0,0.8);
\coordinate (LB)  at (-3.0,-1.4);
\coordinate (RB)  at (-1.0,-1.4);

\coordinate (ML)  at (1.0,-1.4);
\coordinate (MR)  at (3.0,-1.4);
\coordinate (TR)  at (4.0,0.8);


\draw[branch] (C) -- (LU);
\draw[branch] (C) -- (RU);

\draw[branch] (LU) -- (TL);
\draw[branch] (LU) -- (RL);

\draw[branch] (RU) -- (LL);
\draw[branch] (RU) -- (TR);

\draw[branch] (LL) -- (LB);
\draw[branch] (LL) -- (RB);

\draw[branch] (RL) -- (ML);
\draw[branch] (RL) -- (MR);


\draw[infoset] (-2.0,1.9) ellipse [x radius=0.43, y radius=1.55];
\draw[infoset] (2.0,1.9) ellipse [x radius=0.43, y radius=1.55];

\node[num] at (-2.0,1.75) {I};
\node[num] at (2.0,1.75) {II};


\node[chance] at (C) {};

\node[nodecircle] at (LU) {};
\node[nodecircle] at (RU) {};
\node[nodecircle] at (LL) {};
\node[nodecircle] at (RL) {};


\node[num] at (-1.25,5.08) {chance};

\node[num] at (-1.35,4.18) {$\frac{1}{2}$};
\node[num] at (1.35,4.18) {$\frac{1}{2}$};


\node[lab] at (-2.75,2.75) {$T$};
\node[lab] at (-1.35,2.85) {$B$};

\node[lab] at (1.30,2.85) {$l$};
\node[lab] at (2.65,2.75) {$r$};

\node[lab] at (-2.75,0.20) {$T$};
\node[lab] at (-1.35,0.20) {$B$};

\node[lab] at (1.30,0.20) {$l$};
\node[lab] at (2.65,0.20) {$r$};


\node[payoff] at (-4.10,0.15) {$0$\\$0$};

\node[payoff] at (-3.00,-2.05) {$0$\\$6$};

\node[payoff] at (-1.00,-2.05) {$2$\\$0$};

\node[payoff] at (1.00,-2.05) {$2$\\$0$};

\node[payoff] at (3.00,-2.05) {$-2$\\$2$};

\node[payoff] at (4.10,0.15) {$2$\\$0$};

\end{tikzpicture}
\caption{The above game is weakly time-strucutured (under symmetric equivalent classes) but not strong time structured}
  \label{fig:trees2}
\end{figure}
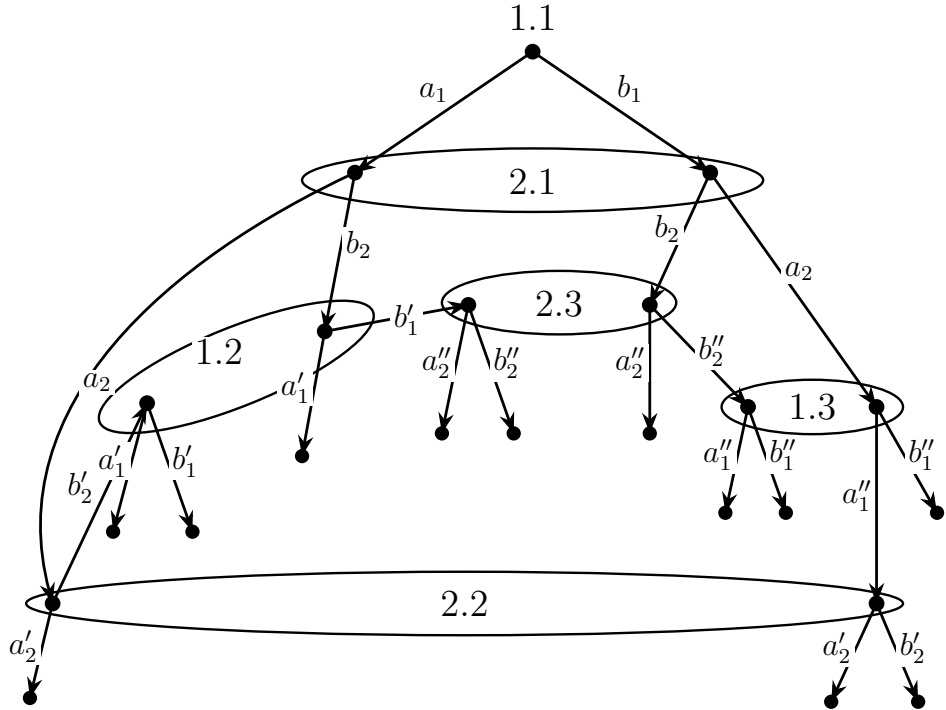
\begin{figure}[htbp]
\begin{tikzpicture}[
    x=1cm,
    y=1cm,
    >=Stealth,
    branch/.style={
        ->,
        draw=black,
        line width=1.05pt,
        line cap=round,
        line join=round
    },
    infoset/.style={
        draw=black,
        line width=0.9pt
    },
    decision/.style={
        circle,
        fill=black,
        inner sep=2.1pt
    },
    terminal/.style={
        circle,
        fill=black,
        inner sep=1.9pt
    },
    player/.style={
        font=\Large,
        inner sep=1pt,
        fill=white
    },
    act/.style={
        font=\large,
        inner sep=1.2pt,
        fill=white
    }
]


\coordinate (N11) at (0,6.35);

\coordinate (N21L) at (-2.35,4.75);
\coordinate (N21R) at (2.35,4.75);

\coordinate (N12L) at (-5.10,1.70);
\coordinate (N12R) at (-2.75,2.65);

\coordinate (N23L) at (-0.85,3.00);
\coordinate (N23R) at (1.55,3.00);

\coordinate (N13L) at (2.85,1.65);
\coordinate (N13R) at (4.55,1.65);

\coordinate (N22L) at (-6.35,-0.95);
\coordinate (N22R) at (4.55,-0.95);


\coordinate (T22L) at (-6.65,-2.20);

\coordinate (T12La) at (-5.55,0.0);
\coordinate (T12Lb) at (-4.50,0.0);

\coordinate (T12Ra) at (-3.05,1.00);
\coordinate (T12Rb) at (-2.05,1.05);

\coordinate (T23La) at (-1.20,1.30);
\coordinate (T23Lb) at (-0.25,1.30);

\coordinate (T23Ra) at (1.55,1.30);

\coordinate (T13La) at (2.55,0.25);
\coordinate (T13Lb) at (3.35,0.25);

\coordinate (T13Rb) at (5.35,0.25);

\coordinate (T22Ra) at (3.95,-2.25);
\coordinate (T22Rb) at (5.10,-2.25);


\draw[infoset] (0,4.65) ellipse [x radius=3.05, y radius=0.42];
\node[player] at (0,4.63) {$2.1$};

\begin{scope}[rotate around={22:(-3.92,2.18)}]
  \draw[infoset] (-3.92,2.18) ellipse [x radius=1.95, y radius=0.52];
\end{scope}
\node[player] at (-4.15,2.38) {$1.2$};

\draw[infoset] (0.35,3.03) ellipse [x radius=1.55, y radius=0.42];
\node[player] at (0.35,3.02) {$2.3$};

\draw[infoset] (3.70,1.65) ellipse [x radius=1.20, y radius=0.36];
\node[player] at (3.70,1.65) {$1.3$};

\draw[infoset] (-0.90,-0.95) ellipse [x radius=5.80, y radius=0.42];
\node[player] at (-0.90,-0.95) {$2.2$};


\node[decision] at (N11) {};
\node[decision] at (N21L) {};
\node[decision] at (N21R) {};
\node[decision] at (N12L) {};
\node[decision] at (N12R) {};
\node[decision] at (N23L) {};
\node[decision] at (N23R) {};
\node[decision] at (N13L) {};
\node[decision] at (N13R) {};
\node[decision] at (N22L) {};
\node[decision] at (N22R) {};


\node[player] at ($(N11)+(0,0.45)$) {$1.1$};

\draw[branch] (N11) -- node[act, above left, pos=0.45] {$a_1$} (N21L);
\draw[branch] (N11) -- node[act, above right, pos=0.45] {$b_1$} (N21R);


\draw[branch]
(N21L) to[out=-155,in=105]
node[act, below, pos=0.55] {$a_2$}
(N22L);

\draw[branch] (N21L) -- node[act, right, pos=0.45] {$b_2$} (N12R);

\draw[branch] (N21R) -- node[act, left, pos=0.42] {$b_2$} (N23R);
\draw[branch] (N21R) -- node[act, right, pos=0.42] {$a_2$} (N13R);


\draw[branch] (N22L) -- node[act, above left, pos=0.48] {$b'_2$} (N12L);
\draw[branch] (N22L) -- node[act, left, pos=0.45] {$a'_2$} (T22L);

\draw[branch] (N12L) -- node[act, left, pos=0.45] {$a'_1$} (T12La);
\draw[branch] (N12L) -- node[act, right, pos=0.45] {$b'_1$} (T12Lb);

\draw[branch] (N12R) -- node[act, left, pos=0.45] {$a'_1$} (T12Ra);
\draw[branch] (N12R) -- node[act, right, pos=0.45] {$b'_1$} (N23L);

\draw[branch] (N23L) -- node[act, left, pos=0.45] {$a''_2$} (T23La);
\draw[branch] (N23L) -- node[act, right, pos=0.45] {$b''_2$} (T23Lb);

\draw[branch] (N23R) -- node[act, left, pos=0.45] {$a''_2$} (T23Ra);
\draw[branch] (N23R) -- node[act, right, pos=0.45] {$b''_2$} (N13L);

\draw[branch] (N13L) -- node[act, left, pos=0.45] {$a''_1$} (T13La);
\draw[branch] (N13L) -- node[act, right, pos=0.45] {$b''_1$} (T13Lb);

\draw[branch] (N13R) -- node[act, left, pos=0.45] {$a''_1$} (N22R);
\draw[branch] (N13R) -- node[act, right, pos=0.45] {$b''_1$} (T13Rb);

\draw[branch] (N22R) -- node[act, left, pos=0.45] {$a'_2$} (T22Ra);
\draw[branch] (N22R) -- node[act, right, pos=0.45] {$b'_2$} (T22Rb);


\node[terminal] at (T22L) {};
\node[terminal] at (T12La) {};
\node[terminal] at (T12Lb) {};
\node[terminal] at (T12Ra) {};
\node[terminal] at (T23La) {};
\node[terminal] at (T23Lb) {};
\node[terminal] at (T23Ra) {};
\node[terminal] at (T13La) {};
\node[terminal] at (T13Lb) {};
\node[terminal] at (T13Rb) {};
\node[terminal] at (T22Ra) {};
\node[terminal] at (T22Rb) {};

\end{tikzpicture}
\caption{The above game is a two player game with perfect recall that is neither weakly nor strongly time-strucutured}
  \label{fig:trees3}
\end{figure}

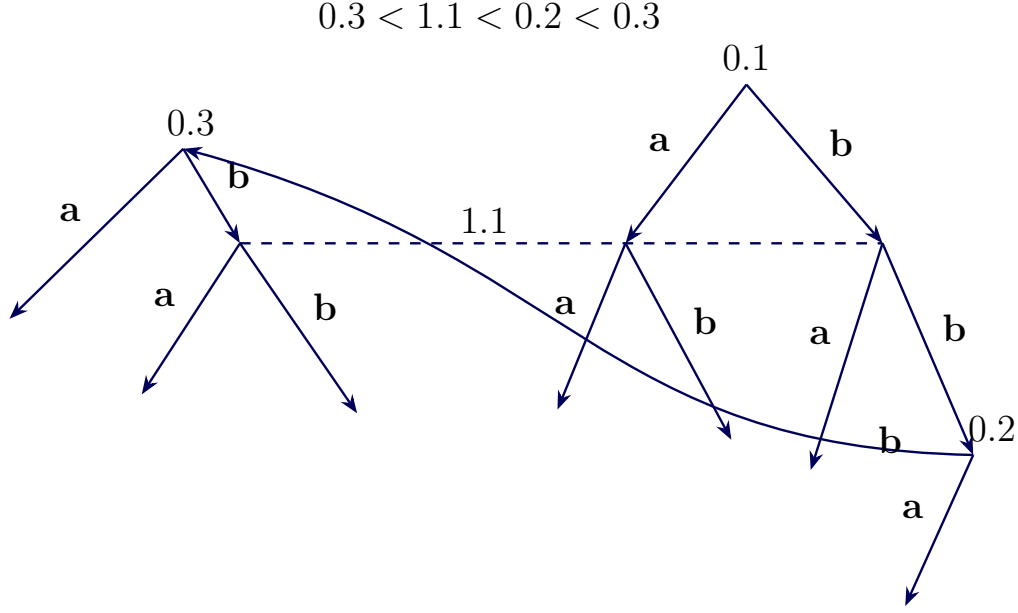
\begin{figure}[htbp]
  \centering
\begin{tikzpicture}[
    x=1cm,
    y=1cm,
    >=Stealth,
    line/.style={->, line width=0.9pt, draw=blue!35!black},
    dashnoarrow/.style={
        dashed,
        line width=0.9pt,
        draw=blue!35!black,
        dash pattern=on 4pt off 4pt
    },
    lab/.style={font=\Large\bfseries, text=black},
    num/.style={font=\Large\bfseries, text=black},
    leftlab/.style={lab, pos=0.50, xshift=-10pt, yshift=8pt},
    rightlab/.style={lab, pos=0.50, xshift=10pt, yshift=8pt}
]

\node[num] at (6.85,6.55) {$0.3 < 1.1 < 0.2 < 0.3$};

\coordinate (Ltop)      at (2.80,4.80);
\coordinate (Lleft)     at (0.50,2.55);
\coordinate (Lmid)      at (3.55,3.55);
\coordinate (Lmidleft)  at (2.25,1.55);
\coordinate (Lmidright) at (5.10,1.30);

\coordinate (Rtop)      at (10.25,5.65);
\coordinate (Rleft)     at (8.65,3.55);
\coordinate (Rright)    at (12.05,3.55);

\coordinate (Rlefta)    at (7.75,1.35);
\coordinate (Rleftb)    at (10.05,0.95);

\coordinate (Rrighta)   at (11.10,0.55);
\coordinate (Rbottom)   at (13.25,0.75);

\coordinate (Rbottoma)  at (12.35,-1.25);

\node[num] at ($(Ltop)+(0.10,0.35)$) {$0.3$};

\draw[line] (Ltop) -- node[leftlab]  {a} (Lleft);
\draw[line] (Ltop) -- node[rightlab] {b} (Lmid);

\draw[line] (Lmid) -- node[leftlab]  {a} (Lmidleft);
\draw[line] (Lmid) -- node[rightlab] {b} (Lmidright);

\node[num] at ($(Rtop)+(0,0.35)$) {$0.1$};

\draw[line] (Rtop) -- node[leftlab]  {a} (Rleft);
\draw[line] (Rtop) -- node[rightlab] {b} (Rright);

\draw[line] (Rleft) -- node[leftlab]  {a} (Rlefta);
\draw[line] (Rleft) -- node[rightlab] {b} (Rleftb);

\draw[line] (Rright) -- node[leftlab]  {a} (Rrighta);
\draw[line] (Rright) -- node[rightlab] {b} (Rbottom);

\node[num] at ($(Rbottom)+(0.25,0.35)$) {$0.2$};

\draw[line] (Rbottom) -- node[leftlab] {a} (Rbottoma);

\draw[dashnoarrow] (Lmid) -- (Rright);
\node[num] at ($(Lmid)!0.38!(Rright)+(0,0.30)$) {$1.1$};

\draw[line]
    (Rbottom)
    .. controls (8.25,0.85) and (8.00,3.40) ..
    (Ltop);

\node[lab] at (12.15,0.95) {b};

\end{tikzpicture}
\caption{The above is a game that is not weakly time-structured and that violates perfect recall}
  \label{fig:trees}
\end{figure}

\subsection*{Debreu‑style continuous representation}
\begin{theorem}
Let $V$ be a second‑countable metric space and $\prec\subseteq V\times V$
a strict partial order satisfying:
\begin{enumerate}
  \item $\prec$ is acyclic,
  \item $\prec$ is topologically closed in $V\times V$,
  \item there exists a countable order‑dense subset $D\subset V$.
\end{enumerate}
Then there exists a continuous $f:V\to\mathbb R$ such that
$x\prec y\Rightarrow f(x)<f(y)$.
\end{theorem}

\begin{proof}
The proof adapts Debreu's separability argument for preorders to strict
partial orders.

Let $D=\{d_1,d_2,\dots\}$ be order‑dense. For each $d_i$ define

\[
L_i:=\{x\in V: x\prec d_i\},\qquad R_i:=\{x\in V: d_i\prec x\}.
\]

Because the graph of $\prec$ is closed, the sets $L_i$ and $R_i$ are
open: if $x\prec d_i$ then by closedness there exists a neighborhood $U$
of $x$ such that for all $u\in U$ we still have $u\prec d_i$. Thus $L_i$
is open; similarly for $R_i$.

Enumerate rationals $\mathbb Q=\{q_1,q_2,\dots\}$. We will construct open
sets $U_{q_k}$ recursively so that:
\begin{enumerate}[label=(\roman*)]
  \item $U_{q_k}\subseteq U_{q_{k'}}$ whenever $q_k<q_{k'}$,
  \item for every $x\prec y$ there exist rationals $q<q'$ with $x\in U_q$
    and $y\notin\overline{U_q}$.
\end{enumerate}

Initialize $U_{q_1}=\varnothing$. At stage $k$, suppose $U_{q_j}$ has been
defined for $j<k$ satisfying monotonicity. Consider the set of ordered
pairs $(x,y)$ not yet separated by earlier $U_{q_j}$. For each such pair
pick $d_i\in D$ with $x\prec d_i\prec y$ (order‑density). Using second
countability, choose for each such $d_i$ a small open neighborhood $W_i$
contained in $R_i$ and disjoint from the closure of the current
$U_{q_{k-1}}$. Let $U_{q_k}$ be the union of $U_{q_{k-1}}$ with finitely
many such neighborhoods $W_i$ chosen so that the new $U_{q_k}$ separates
a finite but increasing set of previously unseparated pairs. Because the
rationals are countable and second countability allows us to refine
neighborhoods, this recursive construction can be carried out so that
every ordered pair is eventually separated at some stage.

Define $f(x)=\inf\{q\in\mathbb Q: x\in U_q\}$. By construction $f(x)<f(y)$
whenever $x\prec y$. Standard refinements of the construction (shrinking
neighborhoods and using second countability) ensure $f$ is continuous.
\end{proof}

\subsection*{Borel measurable labeling in Polish spaces}
\begin{theorem}
Let $G=(V,E)$ be a topological informational digraph with $V$ Polish and
$E$ Borel. Suppose the quotient of $V$ by strongly connected components
is acyclic and the quotient space is standard Borel. Then there exists a
Borel measurable $f:V\to\mathbb R$ with $(u,v)\in E\Rightarrow f(u)<f(v)$.
\end{theorem}

\begin{proof}[Proof]
Collapse SCCs to obtain a standard Borel quotient $\widetilde V$ with a
Borel acyclic partial order. For nonempty Borel subsets of $\widetilde V$
the set of minimal elements is Borel. Apply Kuratowski–Ryll‑Nardzewski to
select minimal elements in a Borel way and enumerate $\widetilde V$ by
iteratively removing selected minimal elements. Assign rational labels in
increasing order to the enumeration to obtain a Borel order embedding
$\widetilde f:\widetilde V\to\mathbb Q$. Lift to $f:V\to\mathbb R$ by
$f(v)=\widetilde f([v])$.
\end{proof}

\subsection*{Finite approximation and limit uniqueness}
\begin{proposition}
Let $G=(V,E)$ be a countable informational digraph admitting an increasing sequence of finite (or otherwise simpler) subgraphs

\[
G_1 \subseteq G_2 \subseteq \cdots
\]

whose union is \(G\), i.e.

\[
\bigcup_{n=1}^\infty G_n = G.
\]
 Suppose each $G_n$ has a unique transitive reduction
and there exists $L$ such that every walk in $G$ has length at most $L$.
Then the transitive reduction of $G$ is unique.
\end{proposition}

\begin{proof}
For any pair $(u,v)$ reachability depends only on vertices within graph
distance $L$ of $u$ and $v$. For sufficiently large $n$ these vertices lie
in $V_n$, so the reachability relation stabilizes and the minimality
decisions for edges are determined at finite stage. The union of finite
stage minimal edge sets yields the unique transitive reduction of $G$.
\end{proof}

\subsection*{Algorithm correctness for recursively presented graphs}
\begin{proposition}
If $G$ is finite then the enumeration algorithm that marks edges redundant
when an intermediate witness is discovered terminates and returns the
unique transitive reduction. If $G$ is countable and the transitive
reduction is finite and recursively enumerable, the algorithm stabilizes
and returns it.
\end{proposition}

\begin{proof}
In the finite case enumeration completes and the algorithm performs the
standard finite transitive reduction test, returning the unique minimal
edge set. In the countable case, if the true transitive reduction $F$ is
finite and recursively enumerable, after a finite stage all edges of $F$
and witnesses for redundancy of other edges are discovered; thereafter
the candidate set stabilizes and the algorithm returns $F$.
\end{proof}

\subsection*{Measurable selection lemmas used}
\begin{theorem}[Kuratowski–Ryll‑Nardzewski]
Let $X$ be a standard Borel space and $Y$ a Polish space. If
$F:X\to 2^Y$ has nonempty closed values and Borel graph then there exists
a Borel selector $s:X\to Y$ with $s(x)\in F(x)$ for all $x$.
\end{theorem}

\begin{remark}
This theorem is used to select Borel representatives of information sets
and to choose minimal elements in Borel acyclic orders.
\end{remark}

\subsection*{Properties of continuum tree construction}
\begin{proposition}
Let $g:[0,1]\to\mathbb R_{\ge0}$ be continuous with $g(0)=g(1)=0$ and
define $d_g(s,t)=g(s)+g(t)-2\min_{u\in[s\wedge t,s\vee t]}g(u)$. Then $d_g$
is a pseudo‑metric on $[0,1]$, and the quotient $T_g=[0,1]/\!\sim$ with
$s\sim t$ iff $d_g(s,t)=0$ is a compact metric real tree. The projection
$\pi:[0,1]\to T_g$ is continuous.
\end{proposition}

\begin{proof}[Proof]
Continuity of $g$ implies $d_g$ is continuous and satisfies the triangle
inequality (standard in the literature on real trees). The quotient of a
compact space by a closed equivalence relation is compact; $d_g$ induces a
metric on equivalence classes. The unique geodesic property follows from
the tree structure encoded by $g$.
\end{proof}

\section*{Acknowledgements}
I thank Mr Wang Jing for his support with the diagrams.

\end{document}